\documentclass{amsart}

\usepackage{amsmath}

\usepackage{cite}
\usepackage{amsthm,thmtools}
\usepackage{amssymb}
\usepackage{ytableau}
\usepackage[mathscr]{eucal}
\usepackage{diagrams}
\usepackage{tikz-cd}
\usepackage{hyperref}
\usepackage{enumitem}
\usepackage{stmaryrd}
\usepackage{microtype}
\usepackage{tabularx}
\usepackage{array}
\usepackage[english]{babel}
\usetikzlibrary{babel}

\newcommand{\wis}{\mathrm}

\newcommand{\BB}{\mathbb{B}}
\newcommand{\FF}{\mathbb{F}}
\newcommand{\NN}{\mathbb{N}}

\newcommand{\RR}{\mathbb{R}}

\newcommand{\ZZ}{\mathbb{Z}}

\newcommand{\C}{\mathcal{C}}
\newcommand{\D}{\mathcal{D}}
\newcommand{\F}{\mathcal{F}}

\newcommand{\Oscr}{\mathcal{O}}
\newcommand{\B}{\mathcal{B}}
\newcommand{\op}{\mathrm{op}}
\newcommand{\mon}{\mathbb{N}^{\times}_+}

\let\M\relax
\DeclareMathOperator{\M}{M}

\DeclareMathOperator{\Hom}{Hom}

\newcommand{\spec}{\wis{Spec}}

\newcommand{\sets}{\mathsf{Sets}}
\newcommand{\setswith}[1]{#1\text{-}\sets}
\newcommand{\et}{\mathsf{Et}}
\newcommand{\etwith}[1]{#1\text{-}\et}

\newcommand{\sh}{\mathsf{Sh}}
\newcommand{\psh}{\mathsf{PSh}}

\newcommand{\ns}{\mathrm{ns}}

\let\max\relax
\newcommand{\max}{\mathrm{max}}

\newtheorem{definition}{Definition}
\newtheorem{proposition}[definition]{Proposition}

\newtheorem{theorem}[definition]{Theorem}
\newtheorem{corollary}[definition]{Corollary}
\newtheorem{lemma}[definition]{Lemma}
\newtheorem{example}[definition]{Example}

\tikzset{
b/.style={bend left=10},
bb/.style={bend left},
cl/.style={outer sep=-1pt},
}

\let\phi\varphi
\let\theta\vartheta

\newcommand{\lattice}[4]{
\foreach \x in {#1,...,#2}{
      \foreach \y in {#3,...,#4}{
        \node[draw,circle,inner sep=2pt,fill] at (\x,\y) {};
      }
    }
}

\DeclareRobustCommand\longtwoheadrightarrow{\relbar\joinrel\twoheadrightarrow}

\title[Topological groupoid representing presheaf topos on a monoid]{A topological groupoid representing the topos of presheaves on a monoid}
\author{Jens Hemelaer}
\thanks{The author is a Ph.D.\ fellow of the Research Foundation -- Flanders (FWO)}

\begin{document}
\begin{abstract}
Butz and Moerdijk famously showed that every (Grothendieck) topos with enough points is equivalent to the category of sheaves on some topological groupoid. We give an alternative, more algebraic construction in the special case of a topos of presheaves on an arbitrary monoid. If the monoid is embeddable in a group, the resulting topological groupoid is the action groupoid for a discrete group acting on a topological space. For these monoids, we show how to compute the points of the associated topos.
\end{abstract}
\maketitle
\tableofcontents

\section{Introduction}

In \cite{butz-moerdijk}, Butz and Moerdijk showed that for every (Grothendieck) topos $\mathcal{T}$, we can find a topological groupoid $G$ such that
\begin{equation*}
\mathcal{T} ~\simeq~ \sh(G),
\end{equation*}
where $\sh(G)$ is the category of sheaves on $G$ (also called the classifying topos of $G$). We give an alternative construction in the case that $\mathcal{T} = \setswith{M}$ for $M$ a monoid. Here $\setswith{M}$ is the topos with 
\begin{itemize}
\item as objects the sets $S$ equipped with a left $M$-action;
\item as morphisms the functions $f : S \to S'$ such that $f(m \cdot s) = m \cdot f(s)$ for all $m \in M$, $s \in S$.
\end{itemize}
We do not assume that $M$ is commutative. Note that $\setswith{M} \simeq \psh(M^\op)$, where $M^\op$ is interpreted as a category with one object, with arrows given by the elements of $M^\op$ and composition given by multiplication.

There is some recent interest in the toposes $\setswith{M}$ with $M$ a monoid. 

For example, in \cite{connes-consani}, Connes and Consani introduced the Arithmetic Site, as part of their approach to the Riemann Hypothesis. This Arithmetic Site has as underlying topos $\setswith{\mon}$, where $\mon$ is the monoid of nonzero natural numbers under multiplication. By equipping this topos with a certain sheaf of semirings, they create a geometric framework for studying the Riemann zeta function.

The approach of Connes and Consani builds on the ideas of \cite{manin}, where the concept of ``algebraic geometry over $\FF_1$'' is discussed, where $\FF_1$ is a conjectural mathematical object suggested by Tits \cite{tits} behaving as if it was a field with one element. While there is no actual field with one element, the hope is to construct a theory in which $\spec(\ZZ)$ is very similar to a curve over a finite field. Then maybe Weil's proof of the Riemann Hypothesis for curves over finite fields can eventually be translated to a proof of the actual Riemann Hypothesis. 

Algebraic geometry over $\FF_1$ is connected to the toposes $\setswith{M}$ in the following way. In \cite{deitmar}, Deitmar defined $\FF_1$-schemes in terms of monoids. To\"en and Vaqui\'e later gave another definition in \cite{toen-vaquie-f1}, within a larger theory of schemes over a symmetric monoidal category. It was then shown by Vezzani in \cite[Theorem 36]{vezzani} that the two definitions are equivalent. The affine $\FF_1$-schemes are dual to commutative monoids, and the topos $\setswith{M}$ can then be seen as the topos of quasi-coherent sheaves on the affine space corresponding to $M$. Pirashvili in \cite{pirashvili} elaborated on this point of view, by studying for example the relation between the prime ideals of $M$ and the topos points of $\setswith{M}$.

Very recently, Rogers in \cite{rogers} studied the toposes $\setswith{M}$ from a more categorical point of view. In that paper, it is shown how one could reconstruct the monoid $M$ from the topos $\setswith{M}$. Further, the essential points of the topos $\setswith{M}$ are computed: it turns out that they correspond to the idempotents of $M$, see \cite[Corollary 4.2]{rogers}.

If $A$ is a semigroup, then by freely adjoining a unit, we get a monoid $A_1$. The category of sets with an $A$-action is then equivalent to $\setswith{A_1}$, see \cite[Section 2]{rogers}. In this way, the study of semigroup actions is a special case of the study of monoid actions. On the other hand, the toposes associated to inverse semigroups in \cite{funk-hofstra} are different, see \cite[Section 6]{rogers}.

The family of toposes $\setswith{M}$ is in some sense ``orthogonal'' to the better-understood family of localic toposes, see \cite[Lemma 3.2]{rogers}. This suggests that we can use the toposes $\setswith{M}$ as a source of examples and counterexamples in topos theory. In this paper, we demonstrate this by constructing an example of a hyperconnected geometric morphism that is not surjective on points (see the end of Subsection \ref{ssec:converse-construction}).

In a previous paper \cite{arithmtop}, we studied the topos $\setswith{\M_2^\ns(\ZZ)}$ with
\[
\M_2^\ns(\ZZ) = \{ a \in \M_2(\ZZ) ~:~ \det(a) \neq 0 \}
\]
This topos has an interesting combinatorial structure and there are some connections to number theory.

We hope that the present paper gives more geometrical insight in the toposes $\setswith{M}$. For example, Corollary \ref{cor:points} gives a method for computing the topos-theoretic points. It should be noted that this method only works if $M$ is embeddable in a group (this is the case for example if $M$ is commutative and cancellative). For arbitrary monoids $M$, computing the points of $\setswith{M}$ seems to be more difficult.

The structure of this paper is as follows. 

In Section \ref{sec:equivariant-as-monoid} we study topological spaces $X$ with a continuous right action of a discrete group $G$. We will give some sufficient conditions for when the associated topos $\sh_G(X)$ of $G$-equivariant sheaves on $X$ is equivalent to the topos $\setswith{M}$ for some monoid $M$. More precisely, in Theorem \ref{thm:equivariant-vs-monoid} we will show that if $G$ acts transitively on a basis $\mathcal{B}$ of $X$, then there is a geometric embedding
\begin{equation*}
\begin{tikzcd}
\sh_G(X) \ar[r] & \setswith{M}
\end{tikzcd}
\end{equation*}
with $M = \{ g \in G : U \cdot g \subseteq U \}$. This embedding is an equivalence if and only if $\mathcal{B}$ is a minimal basis. 

If $X$ has a minimal basis, then we show that the topos of sheaves $\sh(X)$ can be described completely in terms of posets, via the duality between Alexandrov-discrete spaces and preorders. Similarly, $\sh_G(X)$ can be described in terms of posets with an order-preserving right $G$-action.

In Section \ref{sec:converse-construction} we give a converse to Theorem \ref{thm:equivariant-vs-monoid}. If $M$ can be embedded in a group, say $M \subseteq G$ with $G$ a group, then we construct a topological space $X_P$ with a continuous right $G$-action such that $\setswith{M} \simeq \sh_G(X_P)$, see Theorem \ref{thm:converse-construction}. In Subsection \ref{ssec:explicit-translations}, we will explicitly write down the $G$-equivariant sheaf on $X_P$ corresponding to a certain $M$-set $S$. Note that Theorem \ref{thm:converse-construction} gives an alternative proof that $\setswith{M}$ is equivalent to the category of sheaves on a topological groupoid (the action groupoid $X_P \rtimes G$), but only in the special case that $M$ is embeddable in a group.

In Section \ref{sec:arbitrary-monoids} we construct a topological groupoid $G$ such that $\setswith{M} \simeq \sh(G)$, now for $M$ an arbitrary monoid. We first define $G$ as a groupoid in the category of posets (we call this an Alexandrov groupoid). It becomes a topological groupoid after equipping $G_0$ and $G_1$ with the Alexandrov topology. The equivalence $\setswith{M} \simeq \sh(G)$ is then shown in Theorem \ref{thm:main-theorem}.

\section{Equivariant spaces described by a monoid}
\label{sec:equivariant-as-monoid}

For a discrete group $G$ with a continuous right action on a topological space $X$, the category of $G$-equivariant sheaves on $X$ is a Grothendieck topos. Johnstone in \cite[Example 2.1.11(c), p.~76]{johnstone-elephant-1} even gives a very concrete Grothendieck site $\Oscr_G(X)$ such that $G$-equivariant sheaves on $X$ correspond to sheaves on $\Oscr_G(X)$.

The category $\Oscr_G(X)$ has as objects the open subsets of $X$, with morphisms given by
\begin{equation}
\Hom(U,V) = \{ g \in G^\op : U \cdot g \subseteq V  \}
\end{equation}
for $U$ and $V$ open subsets. Composition is given by multiplication in $G^\op$, the opposite group of $G$. Further, a sieve
\begin{equation*}
\{ g_i : U_i \to U \}_{i \in I}
\end{equation*}
is a covering sieve if and only if
\begin{equation}
\bigcup_{i \in I} U_i \cdot g_i = U.
\end{equation}
We denote morphisms in $\Oscr_G(X)$ by their corresponding element in $G$, which is unambiguous once the domain and codomain are specified. Note that Johnstone originally defines $\Oscr_G(X)$ for a left $G$-action on $X$, but since a right $G$-action is the same as a left $G^\op$-action, it is easy to make the translation. The reason we work with a right $G$-action here, is that then Theorem \ref{thm:equivariant-vs-monoid} becomes more natural.

In this section, we would like to describe how equivariant sheaves are related to monoid actions. In some special cases, where the group $G$ acts transitively on a basis of open sets for $X$, we will prove that the topos of equivariant sheaves $\sh_G(X)$ is a subtopos of the topos $\setswith{M}$ of left $M$-sets and $M$-equivariant maps, for some monoid $M$. If this basis is minimal, then in fact we have an equivalence
\[
\sh_G(X) ~\simeq~ \setswith{M}.
\]

\subsection{From equivariant sheaves to monoid actions}
Note that for each open subset $U$ we can define the monoid
\begin{equation}
M_U = \Hom(U,U) \subseteq G^\op
\end{equation}
and for each $g \in G$ there is a monoid isomorphism
\begin{equation}
M_U \longrightarrow M_{(U\cdot g)},~ f \mapsto g^{-1}fg.
\end{equation}
(juxtaposition means multiplication in $G$). 

We will now look at the following special case. Suppose that there is a basis $\B$ for the topology on $X$ such that $G$ acts transitively on $\B$. In other words:
\begin{enumerate}
\item if $U \in \B$ then $U\cdot g \in \B$ for each $g \in G$;
\item for each $U,V \in \B$ there is some $g \in G$ such that $V = U\cdot g$.
\end{enumerate}
In this situation, clearly all monoids $M_U$ with $U \in \B$ are isomorphic to each other. Moreover, for a fixed $U \in \B$, we can interpret $M_U$ as a full subcategory of $\Oscr_G(X)$ with $U$ as its only object. Note that every object in $\Oscr_G(X)$ can be covered by objects that are isomorphic to $U$. By the Comparison Lemma \cite[Theorem 2.2.3, p.~547]{johnstone-elephant-2}, the categories of sheaves $\sh(M_U)$ and $\sh(\Oscr_G(X))=\sh_G(X)$ are now equivalent, if we define the covering sieves on the unique object $U$ of $M_U$ to be the sieves
\begin{equation*}
\{ g_i \}_{i \in I} \subseteq M_U = \Hom(U,U)
\end{equation*}
such that
\begin{equation}
U = \bigcup_{i \in I} (U\cdot g_i).
\end{equation}
This turns $\sh_G(X)$ into a subtopos of $\setswith{M_U^\op}$ (the category of sets with a left $M_U^\op$-action, and $M_U^\op$-equivariant maps between them).

The geometric embedding of $\sh_G(X)$ into $\setswith{M_U^\op}$ is an equivalence if and only if the Grothendieck topology on $\setswith{M_U^\op}$ defining $\sh_G(X)$ is trivial. This is the case if and only if for each sieve
\begin{equation*}
\{ g_i \}_{i \in I} \subseteq M_U
\end{equation*}
such that
\begin{equation}
U = \bigcup_{i \in I} (U\cdot g_i),
\end{equation}
we can find an $i \in I$ such that $U\cdot g_i = U$. Equivalently, for \emph{any} covering
\begin{equation}
U = \bigcup_{i \in I} U_i,\qquad \text{ with }U,U_i \in \B,
\end{equation}
there is some $i \in I$ with $U_i = U$. A basis like this is usually called a  minimal basis, because any other basis contains it.

\begin{definition}
Let $\B$ be a basis of open sets for a topological space $X$. Then $\B$ is called \emph{minimal} if and only if 
\[
U = \bigcup_{i \in I} U_i
\]
for $U$, $U_i$ all in $\B$, implies that $U = U_i$ for some $i \in I$.
\end{definition}

This terminology makes it easier to summarize our observations above.

\begin{theorem} \label{thm:equivariant-vs-monoid}
Let $G$ be a group with a continuous right action on a topological space $X$. Suppose that $G$ acts transitively on a basis $\B$ for $X$. Then there is a geometric embedding
\begin{equation*}
\begin{tikzcd}[column sep=large]
\sh_G(X) \ar[r] & \setswith{M}
\end{tikzcd}
\end{equation*}
with $M \cong \{ g \in G : U\cdot g \subseteq U \} \subseteq G$ for some $U \in \B$. The embedding is an equivalence if and only if $\B$ is a \emph{minimal} basis.
\end{theorem}

Note that the above theorem is just an application of the Comparison Lemma. However, an equivalence between $\sh_G(X)$ and $\setswith{M}$ can be useful in practice. For example, it is easy to show that the topos points of $\sh_G(X)$ are given by the set-theoretic quotient
\[
\hat{X} / G
\]
where $\hat{X}$ is the sobrification of $X$. Computing the topos points of $\setswith{M}$ as flat functors $M^\op \to \sets$ can be more difficult.

\subsection{Examples}

\subsubsection{} Consider a discrete group $G$ with a transitive (right) action on a discrete topological space $X$. Take a point $x \in X$ and let $G_x$ be the stabilizer of this point. Then:
\begin{equation*}
\sh_G(X) ~\simeq~ \setswith{G_x}.
\end{equation*}

\subsubsection{} The (transposed) $ax+b$ group 
\begin{equation*}
G = \left\{ \begin{pmatrix}
a & 0 \\
b & 1 
\end{pmatrix} : a \in \RR^\times,~ b\in\RR
 \right\}
\end{equation*}
acts continuously on the real line $X = \RR$ according to the formula
\begin{equation*}
x \cdot \begin{pmatrix}
a & 0 \\
b & 1 
\end{pmatrix} ~=~ ax+b.
\end{equation*}
Here $G$ acts transitively on the open intervals $(y,z)$ for $y,z \in \RR$, and these are a basis for the topology on $\RR$. So we get a geometric embedding
\begin{equation*}
\sh_G(X) ~\hookrightarrow~ \setswith{M},
\end{equation*}
with
\begin{align*}
M &= \left\{ \begin{pmatrix}
  a & 0 \\
  b & 1
  \end{pmatrix} : 0 \leq b \leq 1,~ 0\leq a+b \leq 1
  \right\}
\end{align*}
under multiplication of matrices. The geometric embedding is not an equivalence, because the family of open intervals is not a minimal basis.

\subsubsection{} Consider $I = (0,+\infty) \subseteq \RR$ with as basis of open sets the subsets $[a,+\infty)$ for $a \in I$. Note that this is a minimal basis for some topology on $I$. The discrete group $\RR^\times$ acts by multiplication (on the right). The induced action on basis open sets is transitive, so we get an equivalence
\begin{equation*}
\sh_{\RR^\times}(I) ~\simeq~ \setswith{[1,+\infty)},
\end{equation*}
where $[1,+\infty)$ is seen as a monoid under multiplication.

\subsection{Description in terms of posets}
\label{ssec:in-terms-of-partial-orders}

Theorem \ref{thm:equivariant-vs-monoid} only shows an equivalence of toposes if the topological space $X$ admits a minimal basis. It might seem as if these kind of topological spaces do not occur in practice. After all, if $X$ has a minimal basis, then it is easy to show that $X$ is not $\mathrm{T}_1$ (in particular, not Hausdorff), except if $X$ is discrete.

We claim that, for our purposes, the topological spaces admitting a minimal basis are the topological spaces that are described by a poset. We will use this to reformulate Theorem \ref{thm:equivariant-vs-monoid} to a theorem about posets with a transitive group action.

First we have to discuss some technical aspects.

A topological space for which we can find a minimal basis, is called a \emph{B-space}. One important class of examples is given by \emph{Alexandrov-discrete spaces}, which are topological spaces such that arbitrary intersections of open sets are again open. For an Alexandrov-discrete space $X$, each point $x \in X$ has a smallest open neighborhood $U_x \ni x$, and the family $\{ U_x \}_{x \in X}$ is a basis of open sets. Alexandrov-discrete spaces are dual to preorders: every preorder is the specialization order for some unique Alexandrov-discrete space.

Suppose that $X$ has a minimal basis $\B$. Then for each $U \in \B$ we can take an element 
\begin{equation}
x_U \in U,\quad x \notin \bigcup_{\substack{V \in \B \\ V \subseteq U}} V.
\end{equation}
It is easy to see that $Y = \{ x_U : U \in \B \} \subseteq X$ (with the subspace topology) determines the same locale as $X$. Moreover, for any continuous right $G$-action on $X$, with $G$ a discrete group, we can define an action on $Y$ by setting:
\begin{equation}
x_U \cdot g = x_{U\cdot g}.
\end{equation}
Now it is easy to see that $\Oscr_G(Y) \simeq \Oscr_G(X)$; in particular $\sh_G(Y) \simeq \sh_G(X)$. Moreover, $Y$ is Alexandrov-discrete, with the additional property that its specialization preorder is actually a partial order.

We now use the duality between Alexandrov-discrete spaces and preorders to reformulate Theorem \ref{thm:equivariant-vs-monoid}.

\begin{definition}
For an element $x$ of a poset $P$, we will use the notation
\begin{equation*}
\,\uparrow\! x ~=~ \{ y \in P : y \geq x \}.
\end{equation*}
These sets are a minimal basis for the \emph{Alexandrov topology} (which has upwards closed subsets as open sets). 

We write $X_P$ for the poset $P$ equipped with the Alexandrov topology. Note that $X_P$ is the unique Alexandrov-discrete space with $P$ as its specialization preorder.
\end{definition}

A poset will always be \emph{interpreted as a category} with as objects the elements of the poset, and as morphisms a unique morphism $u \to v$ for each $u,v \in P$ such that $u \leq v$.

\begin{lemma} \label{lmm:space-vs-poset}
Let $P$ be a poset and let $X_P$ be the Alexandrov-discrete space as defined above. Then:
\[
\sh(X_P) ~\simeq~ \psh(P^\op).
\]
\end{lemma}
\begin{proof}
The sets $\uparrow\! x$ are a minimal basis for $X_P$. The inclusion relation is given by
\[
\,\uparrow\! x \,\subseteq\, \,\uparrow\! y \quad\Leftrightarrow\quad y \leq x.
\]
The lemma then follows.
\end{proof}

A sheaf over $X_P$ can equivalently be described as a local homeomorphism $E \to X_P$ (an \emph{\'etale space}). It will turn out that for each such local homeomorphism $E \to X_P$ we can find a poset $Q$ such that $E = X_Q$. This is why we define the following.

\begin{definition} \label{def:etale}
An order-preserving map of posets $\pi: Q \to P$ will be called \emph{\'etale} if for all $q \in Q$ there is an isomorphism
\begin{equation*}
\,\uparrow\! q ~\cong~ \,\uparrow\! \pi(q).
\end{equation*}
In this case, we will say that $\pi : Q \to P$ is an \emph{\'etale poset over $P$}, or sometimes that $Q$ is an \'etale poset over $P$ (the map $\pi$ is then implicit).

Let $\pi : Q \to P$ and $\pi' : Q' \to P$ be two \'etale posets over $P$. Then a \emph{morphism of \'etale posets} $(Q,\pi) \to (Q',\pi')$ is defined to be an order-preserving map $f : Q \to Q'$ such that the diagram
\[
\begin{tikzcd}[column sep = small]
Q \ar[rr,"{f}"] \ar[rd,"{\pi\vphantom{'}}"']& & Q' \ar[ld,"{\pi'}"] \\
& P & 
\end{tikzcd}
\]
commutes.

The \emph{category of \'etale posets over $P$} will be denoted by $\et/P$.
\end{definition}

\begin{lemma} \label{lmm:etale-is-topos}
Let $P$ be a poset. Then:
\[
\et/P ~\simeq~ \sh(X_P) ~\simeq~ \psh(P^\op).
\]
\end{lemma}
\begin{proof}
The second equivalence is Lemma \ref{lmm:space-vs-poset}, so we prove the first equivalence.

If $Q \to P$ is an \'etale poset, then we want to show that $X_Q \to X_P$ is \'etale (i.e.~a local homeomorphism). For $x \in X_Q$, we can take as open neighborhood the set $\uparrow\! x$. Now it is easy to check that $\uparrow\! x$ maps homeomorphically onto its image.

Conversely, if $E \to X_P$ is \'etale, then it is easy to see that $E$ is Alexandrov-discrete (arbitrary intersections of open subsets are open). Moreover, its specialization order is a partial order. So we can write $E = X_Q$ for a unique poset $Q$. The induced map $Q \to P$ is an \'etale poset.
\end{proof}

\begin{definition} \label{def:equivariant-etale}
Let $G$ be a discrete group acting on the right on a poset $P$, in an order-preserving way. Then a \emph{$G$-equivariant \'etale poset} over $P$ is an \'etale poset $\pi : Q \to P$ equipped with an order-preserving group action $\alpha : Q \times G \to Q$ such that the diagram
\begin{equation*}
\begin{tikzcd}
Q \times G \ar[rr,"{\alpha}"] \ar[rd,"{t}"'] & & Q \ar[ld,"{\pi}"] \\
& P & 
\end{tikzcd}
\end{equation*} 
commutes, with $t$ defined as $t(q,g) = \pi(q)\cdot g$. The partial order on $Q \times G$ is given by $(q,g) \leq (q',g')$ if and only if $q \leq q'$ and $g = g'$.

A \emph{morphism of $G$-equivariant \'etale posets} $(Q,\pi) \to (Q,\pi')$ is a morphism of \'etale posets $f : Q \to Q'$ such that $f(q\cdot g) = f(q) \cdot g$.
\end{definition}

\begin{proposition} \label{prop:equivariant-etale-is-topos}
Let $G$ be a discrete group acting on the right on a poset $P$. Then the induced right $G$-action on $X_P$ is continuous. Moreover:
\[
\etwith{G}/P ~\simeq~ \sh_G(X_P) ~\simeq~ \psh(\C)
\] 
Here $\C$ is the category with
\begin{itemize}
\item as objects the elements of $P$;
\item morphisms given by $\Hom(p,q) = \{ g \in G^\op : p\cdot g \geq q \}$, with composition given by multiplication in $G^\mathrm{op}$.
\end{itemize}
\end{proposition}
\begin{proof}
If we translate all definitions like we did in the proof of Lemma \ref{lmm:etale-is-topos}, then we see that
\[
\etwith{G}/P ~\simeq~ \sh_G(X).
\]

Let $(\Oscr_G(X),J)$ be the site by Johnstone as described in Section \ref{sec:equivariant-as-monoid}. Then
\begin{equation*}
\sh_G(X) \simeq \sh(\Oscr_G(X),J). 
\end{equation*}
Take the functor $F: \C \to \Oscr_G(X)$ sending $p \in P$ to the open set $\,\uparrow\! p \subseteq X$. Then $F$ is fully faithful, and every object of $\Oscr_G(X)$ can be covered by objects in the image of $F$. So we can apply the Comparison Lemma. The induced Grothendieck topology on $\C$ is the trivial Grothendieck topology. So there is an equivalence of toposes
\[
\psh(\C) ~\simeq~ \sh(\Oscr_G(X),J).
\]
This shows the second equivalence $\sh_G(X_P) \simeq \psh(\C)$.
\end{proof}

\begin{corollary} \label{cor:group-action-on-poset}
Let $G$ be a discrete group with a transitive, order-preserving right action on a poset $P$. Then there is an equivalence of toposes
\begin{equation*}
\etwith{G}/P ~\simeq~ \setswith{M}
\end{equation*}
where $M = \{ g \in G : p\cdot g \leq p \} \subseteq G$ for some element $p \in P$.
\end{corollary}
\begin{proof}
This follows immediately from the equivalence $\etwith{G}/P \simeq \psh(\C)$ from Proposition \ref{prop:equivariant-etale-is-topos}. Indeed, take the full subcategory $\D \subseteq \C$ defined by a single object $p$ in $\C$. Then the inclusion functor $\D \to \C$ is essentially surjective, so it is an equivalence of categories. Since $\Hom(p,p) = M^\op$, we see that
\[
\psh(\C) \simeq \psh(\D) \simeq \setswith{M}.
\]
\end{proof}

Note that Corollary \ref{cor:group-action-on-poset} can be seen as a special case of Theorem \ref{thm:equivariant-vs-monoid}.

\section{Converse construction}
\label{sec:converse-construction}

In the previous subsections we have described some cases where a category of equivariant sheaves on a topological space (or a category of equivariant presheaves on a poset) is equivalent to $\setswith{M}$ for some monoid $M$.

In this subsection, we start from a monoid $M$. The question is now if we can find a topological space $X$ with a continuous right action of a discrete group $G$, such that $G$ works transitively on a minimal basis for $G$, and such that $M \cong \{ g \in G : U\cdot g \subseteq U \}$. Then by Theorem \ref{thm:equivariant-vs-monoid} we have $\setswith{M} \simeq \sh_G(X)$. For this to work, we clearly need that $M$ is embeddable into a group. We will show that this condition is also sufficient.

\subsection{From monoid actions to equivariant sheaves}
\label{ssec:converse-construction}

Let $M$ be a monoid embeddable in a group, say $M \subseteq G$ with $G$ a group. Then we can define a preorder on $G$ as follows. For $g,h \in G$ we set
\begin{equation}
h \geq g \quad\Leftrightarrow\quad \exists\,m \in M,\quad h = mg.
\end{equation}

For $g,h \in G$ we say that $g$ is equivalent to $h$ if and only if $g \leq h$ and $h \leq g$. This is the case if and only if there is a unit $u \in M^\times$ such that $g = uh$. So the set of equivalence classes will be denoted by $M^\times\backslash G$.

The preorder on $G$ induces a partial order on $M^\times\backslash G$, given by
\begin{equation}
[g] \geq [h] \quad\Leftrightarrow\quad \exists\,m \in M,\quad g = mh,
\end{equation}
for $g,h \in G$ and $[g],[h]$ their equivalence classes. Directly from Proposition \ref{prop:equivariant-etale-is-topos} and Corollary \ref{cor:group-action-on-poset} we get:

\begin{theorem}[Converse construction] \label{thm:converse-construction}
Let $M$ be a monoid embeddable into a group. Take an arbitrary group $G$ with $M \subseteq G$. Consider the poset $P = M^\times\backslash G$ as above. There is a transitive, order-preserving right action of $G$ on $M^\times\backslash G$, defined by right multiplication. In the notations of the previous section:
\begin{equation*}
\setswith{M} ~\simeq~ \psh(\C) ~\simeq~ \etwith{G}/P ~\simeq~ \sh_G(X_P).
\end{equation*}
\end{theorem}

Here $X_P$ is the Alexandrov-discrete space with as points the elements of $M^\times \backslash G$, and as basis of open sets
\[
U_x ~=~ \{~ [mx] ~:~ m \in M ~\} ~\subseteq~ M^\times \backslash G
\]
for $x \in G$. The right $G$-action by multiplication is continuous. Moreover, $G$ acts transitively on this basis. So the equivalence
\begin{equation} \label{eq:converse-construction}
\setswith{M} ~\simeq~ \sh_G(X_P)
\end{equation}
is also an immediate corollary of Theorem \ref{thm:equivariant-vs-monoid}.

\begin{corollary} \label{cor:points}
Let $M$ be a monoid embeddable into a group. Take an arbitrary group $G$ with $M \subseteq G$. Let $X_P$ be the topological space with right $G$-action, as defined above. Then the points of the topos $\setswith{M}$ up to isomorphism are given by the set-theoretic quotient
\[
\widehat{X_P}/G\, ,
\]
where $\widehat{X_P}$ denotes the sobrification of $X_P$.
\end{corollary}
\begin{proof}
It is easy to show that the topos points of $\sh_H(Y)$ are given by the set-theoretic quotient $\widehat{Y}/H$, whenever $Y$ is a topological space with a continuous right action of a discrete group $H$. The statement now follows from the equivalence $\setswith{M} \simeq \sh_G(X_P)$ in Theorem \ref{thm:converse-construction}.
\end{proof}

We will now demonstrate Corollary \ref{cor:points} by computing the topos points of $\setswith{M}$ in two cases. We will then use the results to construct an example of a hyperconnected geometric morphism that is not surjective on points.

\begin{example} \label{eg:M1}
Let $M$ be the monoid of finite sequences of real numbers, with multiplication given by concatenation of finite sequences. In other words, $M$ is the free monoid on $|\RR|$ generators. Let $G \supseteq M$ be the free group on the same generators. We will call the real numbers \emph{symbols}, to emphasize that the multiplication in $G$ has nothing to do with multiplication in $\RR$.

We would like to compute the points of the topos $\setswith{M}$. Because $M$ has no nontrivial units, we have $X_P = G$ as a set. A basis of open sets for $X_P$ is given by the subsets
\begin{equation*}
U_x ~=~ \{~ mx ~:~ m \in M  ~\}
\end{equation*} 
for $x \in G$. The elements of the sobrification $\widehat{X_P}$ are by definition the irreducible closed subsets in $X_P$. It is easy to see that a subset $V \subseteq X_P$ is closed if and only if it is downwards closed for the specialization preorder on $X_P$. It is irreducible if and only if for any two $v, v' \in V$ there is some $w \in V$ with $v,v' \leq w$ for the specialization preorder. If $V$ has a maximal element $v_0 \in V$ for the specialization preorder, then
\[
V ~=~ \{~ m^{-1}v_0 ~:~ m \in M ~\}.
\]
If $V$ does not have a maximal element, then we can find an element $v_0 \in V$ such that the reduced form of $v_0$ starts with a generator of $M$. We will call an element like this \emph{semipositive}. If $v$ is semipositive and $v \leq v'$, then $v'$ is semipositive as well. We can now use the irreducibility of $V$ to show that for every $v \in V$ we can find a semipositive element $v' \in V$ such that $v \leq v'$. Now take two semipositive elements $v,v' \in V$. Then there is an element $w \in V$ such that $v,v' \leq w$. But this is only possible if $v \leq v'$ or $v' \leq v$. So we can find an infinite sequence of semipositive elements
\[
v_0 \leq v_1 \leq v_2 \leq v_3 \leq \dots
\]
such that
\[
V ~=~ \{~ m^{-1}v_i ~:~ m \in M,~ i \in \NN   ~\}.
\]
In this case, $V$ is completely determined by an infinite sequence of symbols and formal inverses of symbols, such that the reduced form of the sequence contains only finitely many formal inverses. The group $G$ acts on these sequences by concatenation on the right (if we write the infinite sequence from right to left).

By Corollary \ref{cor:points}, the points of $\setswith{M}$ up to isomorphism are given by the set-theoretic quotient $\widehat{X_P}/G$. It is easy to see that the cardinality of this quotient set is equal to the cardinality of the continuum.
\end{example}

\begin{example} \label{eg:M2}
Let $M$ be the monoid of functions $f: \RR\to \NN$ with finite support (so there are only finitely many $x \in \RR$ with $f(x) \neq 0$) under addition. Then $M$ is the free commutative monoid on $|\RR|$ generators. We can compute the topos points of $\setswith{M}$ using the same methods as in the previous example. It is then easy to show that in this case we can take $\widehat{X_P}$ to be the set of all functions $\RR \to \ZZ \cup \infty$. Two functions $f,g : \RR \to \ZZ \cup \infty$ determine the same point of $\setswith{M}$ if and only if there is some $h : \RR \to \ZZ$ with finite support such that $f = g + h$. Here we use the convention $\infty + z = \infty$ for $z \in \ZZ \cup \infty$. In particular, the cardinality of the set of points of $\setswith{M}$ up to isomorphism is $2^{|\RR|}$.
\end{example}

Let $M_1$ be the monoid from Example \ref{eg:M1}, and let $M_2$ be the monoid from Example \ref{eg:M2}. There is a surjective map
\begin{equation}
\psi ~:~ M_1 ~\longtwoheadrightarrow~ M_2
\end{equation}
sending a finite sequence $s$ of real numbers to the function $f_s$ with $f_s(x)$ equal to the number of times that $x$ appears in $s$. Note that $\psi$ can also be interpreted as the abelianization map.

The map $\psi$ induces a geometric morphism
\begin{equation}
\Psi ~:~ \setswith{M_1} ~\longrightarrow~ \setswith{M_2}.
\end{equation}
This geometric morphism is \emph{hyperconnected} by \cite[Proposition 3.1.(ii)]{johnstone-factorization-I}.
However, by looking at the cardinalities, we see that the induced map on isomorphism classes of points is not surjective.
This gives an example of a hyperconnected geometric morphism that is not surjective on points.

\subsection{Explicit translations}
\label{ssec:explicit-translations}

Let $M$ be a monoid embeddable into a group, and take a group $G$ with $M \subseteq G$. We then constructed a poset $P$ with a transitive right $G$-action such that $M$-sets correspond to $G$-equivariant \'etale posets on $P$.

In this subsection, we will make this correspondence more explicit. So for an $M$-set $S$, we will give an explicit description of the corresponding equivariant \'etale poset over $P$. In order to do this, we have to look back at the proofs in Subsections \ref{ssec:in-terms-of-partial-orders} and \ref{ssec:converse-construction}.

In the proof of Corollary \ref{cor:group-action-on-poset}, we showed that
\[
\setswith{M} ~\simeq~ \psh(\C)
\]
with the category $\C$ as defined in Proposition \ref{prop:equivariant-etale-is-topos}. Here we used the Comparison Lemma. Keeping the proof of the Comparison Lemma in mind, we can explicitly construct the presheaf $\F$ on $\C$ corresponding to some $M$-set $S$.

Recall that the elements of $P$ were given by the elements of $G$, up to left multiplication by a unit. Now label the elements of $P$, by taking for every element $p \in P$ a representative $\sigma(p) \in G$. We will assume that $\sigma(1) = 1$, but note that $\sigma$ is not multiplicative. We will make the identification
\[
\F(p) = S
\]
for all $p \in P$, such that the restrictions along the isomorphisms $\sigma(p) : 1 \longrightarrow p$ are the identity $S \to S,~s \mapsto s$.

Each morphism $g : p \to q$ is part of a commutative diagram
\begin{equation*}
\begin{tikzcd}[row sep=huge,column sep=huge]
p \ar[r,"{g}"] & q \\
1 \ar[u,"{\sigma(p)}","{\cong}"'] \ar[r,"{\alpha}"'] & 1 \ar[u,"{\sigma(q)}"',"{\cong}"]
\end{tikzcd}
\end{equation*}
and we can compute $\alpha$ to be
\begin{equation}
\alpha = \sigma(p) g \sigma(q)^{-1} \in M.
\end{equation}
(since composition is defined by multiplication in $G^\mathrm{op}$). Now for $s \in \F(p)$ we know that the restriction of $s$ along $g$ is
\begin{equation}
(\sigma(p) g \sigma(q)^{-1}) \cdot s \in \F(q)
\end{equation}
with the left $M$-action given by the identification $\F(p) = \F(q) = S$.

In particular, for $s \in \F(1)$, the restriction along $m : 1 \to 1$ is given by $m \cdot s$.

Now we will describe the $G$-equivariant \'etale space $\pi : E \to X_P$ associated to $\F$. Recall that a basis for the topology on $X_P$ is given by the subsets
\begin{equation}
\,\uparrow\! p ~=~ \{ q \in X_P : q \geq p \}.
\end{equation}
The fiber $\pi^{-1}(p)$ for $p \in X_P$ is given by definition by the stalk
\begin{equation}
\F_p ~=~ \varinjlim_{q \leq p} \F(q),
\end{equation}
where the transition morphisms $\F(q) \to \F(q')$ with $q \leq q' \leq p$ are given by the restriction along $q' \stackrel{1}{\longrightarrow} q$. Since $p$ is maximal in this filtered diagram, we have
\begin{equation}
\F_p ~=~ \varinjlim_{q \leq p} \F(q) ~\simeq~ \F(p) = S.
\end{equation}
As a set, we can describe $E$ as
\begin{equation}
E ~\simeq~ \bigsqcup_{p \in X_P} \! \F_p ~\simeq~ \bigsqcup_{p \in X_P} \! S.
\end{equation}
We will write the elements of $E$ as couples $(p,s)$ with $p \in X_P$ and $s \in \F(p)$.
A basis for the topology is given by the subsets
\begin{equation}
s(\uparrow\! p) ~=~ \left\{ (q,s_q) : q \geq p  \right\} ~=~ \left\{ (q,\sigma(q)\sigma(p)^{-1} \cdot s) : q \geq p  \right\}
\end{equation}
for each $(p,s) \in E$. The right $G$-action on $E$ is induced by the isomorphisms
\begin{equation}
p \cdot g \stackrel{g^{-1}}{\longrightarrow} p
\end{equation}
for $p \in X_P$ and $g \in G$. In terms of the $M$-action, the corresponding restriction morphism is
\begin{align}
(g^{-1})^\ast ~:~ \F(p) \longrightarrow \F(p \cdot g) 
\end{align}
defined as $(g^{-1})^\ast s = \sigma(p \cdot g) g^{-1} \sigma(p)^{-1} \cdot s$. Now from \cite[Example 2.1.11(c), p.~76]{johnstone-elephant-1}, we know that the right $G$-action on $E$ is given by
\begin{equation}
(p,s_p) \cdot g ~=~ (p \cdot g, ((g^{-1})^\ast s)_{p \cdot g}).
\end{equation}

In terms of the left $M$-action:
\begin{equation} \label{eq:group-action-vs-monoid-action}
(p,s_p) \cdot g ~=~ (p \cdot g, \sigma(p \cdot g) g^{-1} \sigma(p)^{-1} \cdot s_p).
\end{equation}
Because $\sigma(p)g$ and $\sigma(p \cdot g)$ represent the same element in $P$, we can write $g$ as
\begin{equation} \label{eq:def-ug}
g = \sigma(p)^{-1} u_g \sigma(p \cdot g)
\end{equation}
for a unique $u_g \in M^\times$. With this notation, we have
\begin{equation} \label{eq:G-action}
(p,s_p) \cdot g = (p \cdot g, u_g^{-1} \cdot s_p).
\end{equation}
In particular,
\begin{equation}
(p,s_p) \cdot \sigma(p)^{-1}\sigma(p \cdot g) = (q,s_p),
\end{equation}
so the elements $T_{p,p \cdot g} = \sigma(p)^{-1}\sigma(p \cdot g)$ acts as ``translations''. In contrast, for $\theta_{p,u} = \sigma(p)^{-1} u \sigma(p)$ with $u \in M^\times$, we have $(p,s_p)\cdot \theta_{p,u} = (p,u^{-1}\cdot s_p)$. In this way, we get an action of $M^\times$ on each fiber.
Moreover, the translations $T_{p,p\cdot g}$ preserve the $M^\times$-actions on the fibers, in the sense that
\begin{equation}
\left((p,s_p)\cdot \theta_{p,u} \right) \cdot T_{p,p\cdot g} 
= \left( (p,s_p) \cdot T_{p,p\cdot g}  \right) \cdot \theta_{p\cdot g, u}.
\end{equation}

The specialization preorder on $E$ is defined by
\begin{equation} \label{eq:specialization-preorder-on-E}
(q,s') \geq (p,s) \quad\Leftrightarrow\quad q \geq p \text{ and }s' = \sigma(q)\sigma(p)^{-1} \cdot s
\end{equation}
It is easy to see that the topology on $E$, as defined above, agrees with the Alexandrov topology with respect to this partial order. In this way, we can view $E$ as a $G$-equivariant \'etale space over $X_P$, or as a $G$-equivariant \'etale poset over $P$.

\subsection{Example: \texorpdfstring{$\setswith{\NN}$}{N-sets}}

In this subsection, we consider the monoid $\NN = \{0,1,2,\dots\}$ under addition. Take the embedding $\NN \subseteq \ZZ$.

We apply the construction from the previous subsection to this case, with $M = \NN$ and $G = \ZZ$. Since $\NN^\times = \{0\}$, the poset $M^\times \backslash G$ in this case is just $\ZZ$ with the usual partial order. The right action is given by 
\begin{equation*}
\alpha : \ZZ \times \ZZ \to \ZZ,~ \alpha(n,m) = n+m.
\end{equation*}
With the above group action and partial order, we find:
\begin{equation}
\setswith{\NN} ~\simeq~  \etwith{\ZZ}/(\ZZ,\leq).
\end{equation}
Here we use the notations from Subsection \ref{ssec:in-terms-of-partial-orders}.

Let $S$ be a set with a left $\NN$-action, written additively as $(n,s) \mapsto n+s$. We will describe the equivariant \'etale poset corresponding to $S$, using the results from the previous subsection.

As a set, we have
\begin{equation}
E ~=~ \bigsqcup_{z \in \ZZ} S,
\end{equation}
so the elements of $E$ will be denoted by $(z,s)$ with $z \in \ZZ$ and $s \in S$. The projection $\pi: E \to \ZZ$ is $\pi((z,s)) = z$. The right $\ZZ$-action can be described as
\begin{equation}
(z,s) + a ~=~ (z+a,s)
\end{equation}
by combining (\ref{eq:def-ug}) and (\ref{eq:G-action}). The partial order on $E$ is given by
\begin{equation}
(z',s') \geq (z,s) \quad\Leftrightarrow\quad z'\geq z ~\text{ and }~s' = (z'-z) + s.
\end{equation}

With this description in mind, we can draw the equivariant spaces associated to certain $\NN$-sets. In the figures below, the map $\pi$ corresponds to projection on the $x$-axis and the $\ZZ$-action corresponds to horizontal translation. The partial order can be reconstructed by setting $x \leq y$ if and only if there is a path from $x$ to $y$ going from left to right.

\begin{figure}[!ht]
\begin{tikzpicture}
\lattice{1}{9}{0}{0}
\draw[dashed] (0,0)--(1,0);
\draw (1,0)--(9,0);
\draw[dashed] (9,0)--(10,0);
\end{tikzpicture}
\caption{The trivial $\NN$-set.}
\end{figure}

\begin{figure}[!ht]
\begin{tikzpicture}
\lattice{1}{9}{0}{1}
\draw[dashed] (0,0)--(1,1);
\draw (1,1)--(2,0)--(3,1)--(4,0)--(5,1)--(6,0)--(7,1)--(8,0)--(9,1);
\draw[dashed] (9,1)--(10,0);
\draw[dashed] (0,1)--(1,0);
\draw (1,0)--(2,1)--(3,0)--(4,1)--(5,0)--(6,1)--(7,0)--(8,1)--(9,0);
\draw[dashed] (9,0)--(10,1);
\end{tikzpicture}
\caption{The $\NN$-action on $\ZZ/2\ZZ$ given by addition modulo $2$.}
\end{figure}

\begin{figure}[!ht]
\begin{tikzpicture}
\lattice{1}{9}{0}{2}
\draw[dashed] (0,0)--(1,2);
\draw (1,2)--(2,1)--(3,0)--(4,2)--(5,1)--(6,0)--(7,2)--(8,1)--(9,0);
\draw[dashed] (9,0)--(10,2);
\draw[dashed] (0,1)--(1,0);
\draw (1,0)--(2,2)--(3,1)--(4,0)--(5,2)--(6,1)--(7,0)--(8,2)--(9,1);
\draw[dashed] (9,1)--(10,0);
\draw[dashed] (0,2)--(1,1);
\draw (1,1)--(2,0)--(3,2)--(4,1)--(5,0)--(6,2)--(7,1)--(8,0)--(9,2);
\draw[dashed] (9,2)--(10,1);
\end{tikzpicture}
\caption{The $\NN$-action on $\ZZ/3\ZZ$ given by addition modulo $3$.}
\end{figure}

\begin{figure}[!ht]
\begin{tikzpicture}
\lattice{1}{9}{0}{1}
\draw (1,1)--(2,0) (3,1)--(4,0) (5,1)--(6,0) (7,1)--(8,0);
\draw[dashed] (9,1)--(10,0) (9,0)--(10,0);
\draw[dashed] (0,1)--(1,0) (0,0)--(1,0);
\draw (2,1)--(3,0) (4,1)--(5,0) (6,1)--(7,0) (8,1)--(9,0);
\draw (1,0)--(2,0)--(3,0)--(4,0)--(5,0)--(6,0)--(7,0)--(8,0)--(9,0);
\end{tikzpicture}
\caption{The action $\NN \times \{0,1\} \to \{0,1\}$ given by $(n,x)\mapsto \max(n+x,1)$.}
\end{figure}

\section{Arbitrary monoids} \label{sec:arbitrary-monoids}

Let $M$ be an arbitrary monoid. In this section, we will construct a topological groupoid $G$ such that
\begin{equation*}
\setswith{M} ~\simeq~ \sh(G).
\end{equation*}
For an arbitrary topos $\mathcal{T}$ with enough points, Butz and Moerdijk in \cite{butz-moerdijk} already constructed a topological groupoid $G$ such that $\mathcal{T} \simeq \sh(G)$. However, we hope that our construction, in this special case, will be more practical in applications.

In Section \ref{sec:equivariant-as-monoid} we considered a right action of a discrete group $G$ on a topological space $X$. If $G$ acts transitively on a minimal basis of $X$, then
\begin{equation*} \label{eq:equivalence-action}
\sh_G(X) ~\simeq~ \setswith{M}
\end{equation*}
for $M$ a certain submonoid of $G$. Conversely, if $M \subseteq G$ is a submonoid of a group, then it is easy to find a topological space $X$ with a right $G$-action, such that (\ref{eq:equivalence-action}) holds; this is what we did in Section \ref{sec:converse-construction}.

The difficulty appears when $M$ can not be embedded in a group. For example, consider the commutative monoid $\BB = (\{0,1\},+)$ with $0+0 = 0$ and $1+1=1+0=1$. Then $M$ is not cancellative, so a fortiori it cannot be embedded in a group. Mal'cev in \cite{malcev-counterexample} even gives an example of a (noncommutative) cancellative monoid that cannot be embedded into a group. 

We will circumvent this problem in the following way. Write the arbitrary monoid $M$ as a quotient $M \cong N/\!\sim\,$ such that $N$ can be embedded into a group, say $N \subseteq Z$ with $Z$ a group. Here $\sim$ is a \emph{congruence}: an equivalence relation $\sim~\subseteq N \times N$ satisfying
\begin{equation}
a \sim b \quad\text{and}\quad c\sim d \quad\Rightarrow\quad ac \sim bd
\end{equation}
for $a,b,c,d \in N$ (this is the terminology from e.g.\ \cite{kurokawa}). We then construct a topological groupoid $G$ from $N$, $Z$ and $\sim$.

For $\sim~= \Delta$ the trivial congruence relation given by
\begin{equation}
\Delta = \{ (n,n) : n \in N \} \subseteq N \times N,
\end{equation}
it turns out that we can alternatively describe the topological groupoid in terms of a continuous action of a discrete group on a topological space. So here we are in the same situation as in the previous section.

Each monoid $M$ is a quotient of a free monoid by a congruence. Moreover, free monoids can be embedded in a group (the free group on the same generators). So for each monoid $M$ we can find an $N$ and $\sim$ as above. However, the construction works for any presentation
\begin{equation}
M \cong N/\!\sim,
\end{equation}
($N$ is not necessarily a free monoid). Here are some examples of isomorphisms $M \cong N/\!\sim$ with $N$ embeddable in a group.

\begin{example} \ 
\begin{enumerate}
\item Take $N = \NN = \{0,1,2,\dots \}$ and $Z = \ZZ$ under addition. An example of a congruence is now the equivalence relation generated by
\begin{equation*}
2 + k \sim 5 + k,\quad \text{for all }k \in \NN.
\end{equation*} \vspace{-1em} \label{example:e1}
\item Similarly, we can take $N = \NN$, $Z = \ZZ$ and $k \sim 1$ for all $k \geq 1$. Then $M = N/\!\sim$ is the boolean semiring $\BB$ under addition. This semiring is sometimes called the field with one element.
\item Let $R \neq 0$ be a commutative ring without zero divisors. Then the matrix ring $M = \M_n(R)$ as a monoid under multiplication is not cancellative. So we take a related cancellative monoid
\begin{equation*}
N = \M_n^\ns(R[t]) = \{ m \in \M_n(R[t]) : \det(m) \neq 0 \}.
\end{equation*}
Evaluation in zero defines a multiplicative surjection 
\begin{equation*}
\mathrm{ev}~:~ \M_n^\ns(R[t]) \longtwoheadrightarrow \M_n(R).
\end{equation*}
So we can define the equivalence relation
\begin{equation*}
f \sim g \quad\Leftrightarrow\quad \mathrm{ev}(f) = \mathrm{ev}(g)
\end{equation*}
for $f,g \in N$. Then we get $~\M_n(R)\simeq N/\!\sim$.
\item Similarly, for the cancellative monoid
\[
N = \M_n^\ns(\ZZ) = \{ m \in \M_n(\ZZ) : \det(m) \neq 0 \},
\]
there are surjective multiplicative maps
\[
\pi~:~\M_n^\ns(\ZZ) \longtwoheadrightarrow \M_n(\ZZ/k\ZZ)
\]
for each $k > 0$ (given by reduction modulo $k$). So we can define a congruence $m \sim m' ~\Leftrightarrow~ \pi(m) = \pi(m')$ and then
\begin{equation*}
\M_n(\ZZ/k\ZZ) ~\simeq~ \M_n^\ns(\ZZ)/\!\sim.
\end{equation*}
\end{enumerate}
\end{example}

\subsection{Alexandrov groupoids}

For the topological groupoid that we will associate to the monoid $M = N/\!\sim$, both $G_0$ and $G_1$ will be Alexandrov-discrete spaces, see Subsection \ref{ssec:in-terms-of-partial-orders}. In fact, the specialization preorders of $G_0$ and $G_1$ will be posets, and the topologies on $G_0$ and $G_1$ are the corresponding Alexandrov topologies.

With this in mind, it is more natural to interpret $G$ as a groupoid object in the category of posets. So $G_0$ and $G_1$ are posets, and the maps $s,t,\mu,\iota,e$ are all order-preserving.

\begin{definition}
A groupoid object in the category of posets will be called an \emph{Alexandrov groupoid}.
\end{definition}

The relation between Alexandrov groupoids and topological groupoids will be made clear in Lemma \ref{lmm:alexandrov-grpd-is-topological-grpd} and Lemma \ref{lmm:epo-over-grpd-is-sheaf-over-grpd}.

\begin{lemma} \label{lmm:alexandrov-grpd-is-topological-grpd}
Let $G$ be an Alexandrov groupoid. Equip $G_0$ and $G_1$ with the Alexandrov topology with respect to the partial orders (so the open sets are the upwards closed subsets). Then $G$ is a topological groupoid.
\end{lemma}
\begin{proof}
Follows from the fact that a map is continuous for the Alexandrov topology if and only if it respects the partial order.
\end{proof}

For the following definition, we use the notations from Subsection \ref{ssec:in-terms-of-partial-orders}.

\begin{definition} \label{def:epo-over-grpd}
Let $G$ be an Alexandrov groupoid. Then an \emph{\'etale poset over $G$} is an \'etale poset $\pi : E \to G_0$ together with a right $G$-action
\begin{equation*}
\alpha : E \times_{G_0} G_1 \longrightarrow E
\end{equation*}
preserving the partial order, and satisfying the usual axioms for a groupoid action. A morphism of \'etale posets over $G$ is a function preserving the partial order, the projection maps and the groupoid action.

The category of \'etale posets over $G$ will be denoted by $\et/G$.
\end{definition}

\begin{lemma} \label{lmm:epo-over-grpd-is-sheaf-over-grpd}
Let $G$ be an Alexandrov groupoid. Then the category $\et/G$ of \'etale posets over $G$ is equivalent to the category of sheaves over $G$, where $G$ is seen as a topological groupoid for the Alexandrov topology. 

In particular, $\et/G$ is a topos.
\end{lemma}
\begin{proof}
Note that the category of sheaves over $G$ is by definition equivalent to the category of \'etale spaces $\pi : E \to G_0$ with a continuous groupoid action 
\begin{equation*}
E \times_{G_0}  G_1 \longrightarrow E.
\end{equation*}

If $\pi: E \to G_0$ is an \'etale poset, then it is a local homeomorphism for the Alexandrov topology. The groupoid action on $E$ is continuous, because it respects the partial order. 

Conversely, if $\pi : E \to G_0$ is a local homeomorphism with a continuous $G$-action, then from Lemma \ref{lmm:etale-is-topos} we know that there is a partial order on $E$ such that $\pi : E \to G_0$ is an \'etale poset, and such that the topology on $E$ is the Alexandrov topology for this partial order. Moreover, the $G$-action respects this partial order, so $\pi : E \to G_0$ is an \'etale poset over $G$.
\end{proof}

In this way, we can see an Alexandrov groupoid $G$ as a special kind of topological groupoid. We end this subsection with an example.

\begin{definition}
Let $Z$ be a (discrete) group. Let $P$ be a poset, equipped with a right $Z$-action preserving the partial order. Then the \emph{action groupoid}
\[
G = P \rtimes Z
\]
is the Alexandrov groupoid defined as:
\begin{align*}
\begin{split}
G_0 ~=~ P \qquad&\qquad G_1 ~=~ P \times Z \\
s(p,z) = p \qquad&\qquad t(p,z) = pz \\
\mu((p,z),(pz,z'))= (p,z') \qquad&\qquad e(p) = (p,1) \\
\iota(p,z) = (pz,z^{-1}), \qquad&\qquad 
\end{split}
\end{align*}
with $s,t,\mu,e,\iota$ the source, target, multiplication, unit and inverse maps respectively. The partial order on $G_0$ is the same as on $P$, the partial order on $G_1$ is defined by:
\begin{equation*}
(p,z) \leq (p',z') \quad\Leftrightarrow\quad p \leq p' \quad\text{and}\quad z = z'.
\end{equation*}
\end{definition}

It is easy to check that $P \rtimes Z$ is indeed an Alexandrov groupoid. By construction, \'etale posets over $P \rtimes Z$ are the same as $Z$-equivariant \'etale posets over $P$. So we get:

\begin{proposition} \label{prop:action-groupoid}
Let $Z$ be a (discrete group). Let $P$ be a poset, equipped with a right $Z$-action preserving the partial order. Then there is an equivalence of categories
\begin{equation*}
\et/(P \rtimes Z) ~\simeq~ \etwith{Z}/P.
\end{equation*}
\end{proposition}

\subsection{Alexandrov groupoid associated to an arbitrary monoid}
\label{ssec:the-alexandrov-groupoid}

We will continue using the notations above, so $M$ is an arbitrary monoid written as a quotient $N/\!\sim$ with $N \subseteq Z$ a submonoid of a group $Z$. We want an Alexandrov groupoid $G$ such that
\begin{equation*}
\setswith{M} ~\simeq~ \sh(G).
\end{equation*}
As underlying (set-theoretic) groupoid, we take
\begin{equation} \label{eq:groupoid-first}
\begin{tikzcd}
G~:~\quad G_1 \ar[r,shift left=1,"{s}"] \ar[r,shift right=1,"{t}"'] & G_0
\end{tikzcd}
\end{equation}
with 
\begin{gather}
\begin{split}
G_1 ~=~ (Z\times Z)\,/\,\mathcal{R}_1 \\
G_0 ~=~ Z\,/\,\mathcal{R}_0.
\end{split}
\end{gather}
The equivalence relation $\mathcal{R}_1$ is given by
\begin{gather}
(i',j') \,\mathcal{R}_1\, (i,j) ~\Leftrightarrow~ \exists u,v \in N^\times,~ u \sim v,~ i' = ui,~j' = vj.
\end{gather}
Similarly, $\mathcal{R}_0$ is defined as
\begin{gather}
i' \,\mathcal{R}_0\, i ~\Leftrightarrow~ \exists u \in N^\times,~ i' = ui.
\end{gather}
It is easy to check that now
\begin{equation} \label{eq:R2}
G_1 \times_{G_0} G_1 ~\simeq~ (Z\times Z \times Z)\,/\,\mathcal{R}_2
\end{equation}
where $\mathcal{R}_2$ is defined as
\begin{gather}
\begin{split}
(i',j',k') \,\mathcal{R}_2\, (i,j,k) ~\Leftrightarrow~ \\ \exists u,v,w \in N^\times,~ u \sim v \sim w,~ i' = ui,~ j' = vj,~k' = wk.
\end{split}
\end{gather}
In the identification from (\ref{eq:R2}), an element $(i,j,k)$ in $(Z \times Z \times Z) \,/\,\mathcal{R}_2$ corresponds to the element $((i,j),(j,k))$ in $G_1 \times_{G_0} G_1$. We will keep using this identification throughout the section.

The source and target maps for $G$ are given by
\begin{gather}
\begin{split}
s(i,j) = i \\
t(i,j) = j
\end{split}
\end{gather}
(it is easy to see that this is well-defined). Further, the multiplication $\mu$, the inverse $\iota$ and the unit $e$ are given by
\begin{align} \label{eq:mu-iota-e}
\begin{split}
\mu ~:~ G_1 \times_{G_0}  G_1
\longrightarrow G_1,&\quad \mu(i,j,k) = (i,k) \\
\iota ~:~ G_1 \longrightarrow G_1,&\quad \iota(i,j) = (j,i) \\
e ~:~ G_0 \longrightarrow G_1,&\quad e(z) = (z,z).
\end{split}
\end{align}

We will now equip $G_1$ and $G_0$ with a partial order. This induces a partial order on $G_1 \times_{G_0} G_1$ as follows:
\begin{equation}
(i,j,k) \leq (i',j',k') \quad\Leftrightarrow\quad (i,j)\leq(i',j') ~\text{ and }~ (j,k)\leq(j',k').
\end{equation}
After introducing the partial orders, we will check that $s,t,\mu,\iota,e$ respect the partial orders. In other words, $G$ is an Alexandrov groupoid.

\begin{definition}
The partial order on $G_0$ is given by
\begin{equation*}
i' \geq i \quad\Leftrightarrow\quad \exists n \in N,~ i' = ni.
\end{equation*}
The partial order on $G_1$ is given by
\begin{equation*}
(i',j') \geq (i,j) \quad\Leftrightarrow\quad \exists a,b \in N,~a\sim b,~ i' = ai,~ j' = bj.
\end{equation*}
\end{definition}

Now we can compute that the induced partial order on $G_1 \times_{G_0} G_1$ is given by
\begin{equation}
(i',j',k') \geq (i,j,k) ~\Leftrightarrow~ \exists a,b,c \in N,~ a \sim b \sim c,~i' = ai,~j' = bj,~k'=ck.
\end{equation}

\begin{lemma}
With the definitions as above, $G$ is an Alexandrov groupoid.
\end{lemma}
\begin{proof}
We have to show that $s,t,\mu,\iota,e$ respect the partial orders. For example, if $(i',j',k') \geq (i,j,k)$, then we have to show that $(i',k') \geq (i,k)$. Take elements $a,b,c \in N$ with $a \sim b \sim c$, such that $i' = ai$, $j' = bj$, $k' = ck$. Then $a \sim c$ by transitivity, so $(i',k') \geq (i,k)$. This shows that $\mu$ respects the partial order. The proofs for $s,t,\iota,e$ are similar.
\end{proof}

We claim that the category of \'etale posets over $G$ is equivalent to the category of left $M$-sets. We first show this in the case $M = N$, i.e.~for the trivial congruence.

\begin{proposition} \label{prop:groupoid-trivial-case}
Let $N \subseteq Z$ be a submonoid of a group $Z$, and consider the trivial congruence $\Delta$ defined by 
\begin{equation*}
a \sim b \quad\Leftrightarrow\quad a = b.
\end{equation*}
Let $G$ be the Alexandrov groupoid associated to $(N,Z,\Delta)$ introduced above. Then there is an equivalence of categories
\begin{equation*}
\et/G ~\simeq~ \setswith{N}.
\end{equation*}
\end{proposition}
\begin{proof}
From Theorem \ref{thm:converse-construction}, we know that
\[
\setswith{N} ~\simeq~ \etwith{Z}/P
\]
where the right $Z$-action on $P = N^\times \backslash Z$ is given by multiplication. From Proposition \ref{prop:action-groupoid}, we then know that
\[
\setswith{N} ~\simeq~ \et/(P \rtimes Z).
\] 
We claim that $G$ is isomorphic to $P \rtimes Z$ as Alexandrov groupoids. The isomorphism is given by the identity
\begin{gather*}
G_0 \longrightarrow (P \rtimes Z)_0 \\
i \mapsto i
\end{gather*}
and by the bijection
\begin{gather} \label{eq:identification-with-action-groupoid}
\begin{split}
G_1 \longrightarrow (P \rtimes Z)_1 \\
(i,j) \mapsto (i,i^{-1}j).
\end{split}
\end{gather}
It is easy to check that these bijections preserve the partial orders and the groupoid structure.
\end{proof}

We still have to prove the generalization of Proposition \ref{prop:groupoid-trivial-case} to arbitrary congruences.

\subsection{Proof of the main theorem}

Let $M$ be an arbitrary monoid, written as
\begin{equation}
M \cong N/\!\sim
\end{equation}
where $N$ is a submonoid of a group $Z$, and $\sim$ is a congruence. Without loss of generality, we assume that the isomorphism is an equality, so $M = N/\!\sim$. For $n \in N$, the equivalence class of $n$ is written as $[n] \in M$.

Let $G$ be the Alexandrov groupoid associated to $(N,Z,\sim)$.
Let $G(\Delta)$ be the Alexandrov groupoid associated to $(N,Z,\Delta)$, where
$\Delta$ is the trivial congruence $a\,\Delta\,b \Leftrightarrow a = b$. Note that $G(\Delta)_0 = G_0$.

Suppose that $\pi: E \to G_0$ is an \'etale poset over $G$, and let
\begin{equation}
\alpha : E \times_{G_0} G_1 \longrightarrow E
\end{equation}
be the groupoid action.

Define an action of the groupoid $G(\Delta)$ according to the formula
\begin{gather}
\begin{split}
\alpha_{\Delta} : E \times_{G_0} G(\Delta)_1 \longrightarrow E \\
\alpha_{\Delta}(x,g) = \alpha(x,\phi(g)),
\end{split}
\end{gather}
where $\phi : G(\Delta)_1 \longtwoheadrightarrow G_1$ is the natural order-preserving projection. In this way, $E$ can be interpreted as an \'etale poset over $G(\Delta)$. For two \'etale posets $E$, $E'$ over $G$ and a function $f : E \to E'$, we can check that $f$ is a morphism in $\et/G$ if and only if $f$ is a morphism in $\et/G(\Delta)$. So we can see $\et/G$ as a full subcategory of $\et/G(\Delta)$.

Similarly, each left $M$-set $S$ can be interpreted as a left $N$-set, by defining the $N$-action to be
\begin{equation}
n \cdot s ~=~ [n] \cdot s
\end{equation}
for all $n \in N$ and $s \in S$, where $[n]$ denotes the equivalence class of $n$ in $M = N/\!\sim$.

We are now in the following situation:
\begin{equation}
\begin{tikzcd}
\et/G(\Delta) \ar[draw=none]{r}[sloped,auto=false]{\scalebox{1.5}{$\simeq$}} &
\setswith{N} \\
\et/G \ar[draw=none]{u}[sloped,auto=false]{\scalebox{1.5}{$\subseteq$}} &
\setswith{M} \ar[draw=none]{u}[sloped,auto=false]{\scalebox{1.5}{$\subseteq$}}
\end{tikzcd}
\end{equation}
It remains to show that the $\et/G$ and $\setswith{M}$ determine the same full subcategories.

So let $S$ be an $N$-set. Let $\pi: E \to G_0$ be the corresponding \'etale poset over $G(\Delta)$, from the equivalence
\begin{equation}
\et/G(\Delta) ~\simeq~ \setswith{N}.
\end{equation}
We then want to show that $E$ comes from some \'etale poset over $G$ if and only if $S$ is an $M$-set, in the sense that
\begin{equation} \label{eq:M-set}
a \sim b \quad\Rightarrow\quad a \cdot s = b \cdot s
\end{equation}
for all $a,b \in N$, $s \in S$ (this is precisely what we need in order to make sure that the $M$-action given by $[n]\cdot s = n \cdot s$ is well-defined).

From the proof of Proposition \ref{prop:groupoid-trivial-case}, we know that we can interpret $E$ as a $Z$-equivariant \'etale poset over $G_0$, where the right action of $Z$ on $G_0=Z/\mathcal{R}_0$ is given by multiplication. Now we can use the explicit description of Subsection \ref{ssec:explicit-translations} to describe $E$ in terms of the $N$-set $S$.

For every $p \in G_0$ we take a representative $\sigma(p) \in G$. Then as in Subsection \ref{ssec:explicit-translations}, we make the identification
\begin{equation}
E ~\simeq~ S \times G_0.
\end{equation}
Under this identification, we get
\begin{align}
\begin{split}
E \times_{G_0} G_1 ~&\simeq~ (S \times G_0) \times_{G_0} G_1 \\
                   ~&\simeq~ S \times G_1.
\end{split}
\end{align}
These are identifications as sets, and the partial order on $E$ induces partial orders on both $S \times G_0$ and $S \times G_1$. For $(s,i),(s',i') \in S \times G_0$ we have:
\begin{gather}
\begin{split}
(s,i) \leq (s',i') \quad\Leftrightarrow\quad i \leq i' \quad\text{and}\quad
s' = \sigma(i')\sigma(i)^{-1} \cdot s,
\end{split}
\end{gather}
see (\ref{eq:specialization-preorder-on-E}). Similarly, for $(s,(i,j)),(s',(i',j'))$ in $S \times G_1$ we see that
\begin{gather}
\begin{split}
(s,(i,j)) \leq (s',(i',j')) \quad\Leftrightarrow&\quad (s,i) \leq (s',i')
\quad\text{and}\quad (i,j) \leq (i',j') \\
\quad\Leftrightarrow&\quad \exists a,b \in N,~a \sim b,~ i' = ai,~ j' = bj,\\ &\quad s' = \sigma(i')\sigma(i)^{-1} \cdot s.
\end{split}
\end{gather}

Using Equation (\ref{eq:group-action-vs-monoid-action}) and the isomorphism (\ref{eq:identification-with-action-groupoid}), we can now write down the groupoid action $\alpha : S \times G_1 \longrightarrow S \times G_0$.
\begin{equation}
\alpha(s,(i,j)) = (j,\sigma(j)j^{-1}i\sigma(i)^{-1} \cdot s).
\end{equation}

Note that $\alpha$ preserves the partial order if and only if
\begin{gather*}
\begin{split}
\exists a,b \in N,~a \sim b,~i' = ai,~j'=bj,~s'=\sigma(i')\sigma(i)^{-1} \cdot s \\
\Rightarrow~ j' \geq j ~\text{ and }~
\sigma(j')(j')^{-1} i' \sigma(i')^{-1} \cdot s' = \sigma(j')\sigma(j)^{-1} \cdot (\sigma(j)j^{-1}i\sigma(i)^{-1} \cdot s).
\end{split}
\end{gather*}
By simplifying the above formula, we see that this is the case if and only if
\begin{gather}
\begin{split}
ai\sigma(i)^{-1} \cdot s 
= b i \sigma(i)^{-1} \cdot s.
\end{split}
\end{gather}
for all $(i,j) \in G_1$, $s \in S$ and $a,b \in N$ such that $a \sim b$.

Equivalently, $\alpha$ preserves the partial order if and only if
\begin{equation}
a \cdot s = b \cdot s
\end{equation}
for all $s \in S$ and $a,b \in N$ such that $a \sim b$, in other words Equation (\ref{eq:M-set}) holds.

Conclusion: the \'etale poset $E$ over $G(\Delta)$, associated to an $N$-set $S$, is an \'etale poset over $G$ if and only if $S$ is an $M$-set. So we proved the following theorem:

\begin{theorem} \label{thm:main-theorem}
Let $M$ be an arbitrary monoid, and write $M$ as a quotient
\begin{equation*}
M \cong N/\!\sim
\end{equation*}
where $N$ is a submonoid of a group $Z$, and $\sim$ is a congruence. Let $G$ be the Alexandrov groupoid constructed in Subsection \ref{ssec:the-alexandrov-groupoid}. Then there is an equivalence
\begin{equation*}
\setswith{M} ~\simeq~ \et/G.
\end{equation*}
In particular, there is an equivalence
\begin{equation*}
\setswith{M} ~\simeq~ \sh(G)
\end{equation*}
where $G$ is seen as a topological groupoid under the Alexandrov topology.
\end{theorem}

The Alexandrov topology on $G_0$ has a basis of open sets consisting of the subsets
\begin{equation}
U_z ~=~ \{~  i \in G_0 ~:~ \exists n \in N,~ i = nz   ~\}
\end{equation}
(it is easy to check that the condition $\exists n \in N,~ i = az$ is independent on the choice of representatives for $z$ and $i$). The Alexandrov topology on $G_1$ has a basis of open sets consisting of the subsets
\begin{equation}
\Pi_{(x,y)} ~=~ \{~  (i,j) \in G_1 ~:~ \exists a,b \in N,~ a\sim b,~ i = ax,~ j = by  ~\}
\end{equation}
(the condition does not depend on the choice of representatives for $x$, $y$, $i$ and $j$).

Note that, with some abuse of notation,~ $U_z = U_1 \cdot z$ ~and~ $\Pi_{(x,y)} = \Pi_{(1,1)} \cdot (x,y)$. So in both cases, the basis of open sets consists of all translations of a single basic open set. In this way, we can think of a basic open set as a ``pattern'' or a ``print''.

\begin{example} \
\begin{enumerate}
\item Take $N = \NN$, $Z = \ZZ$, and congruence $\sim$ generated by
\begin{equation*}
2 + k \sim 5 + k, \qquad \text{for all }k \in \NN.
\end{equation*}
Then $G_1 = \ZZ \times \ZZ$, $G_0 = \ZZ$. The topology on $G_0$ has as open sets the upwards closed sets for the usual partial order on $\ZZ$. The basic open set containing $\Pi_{(x,y)} \subseteq G_1$ can be drawn as:
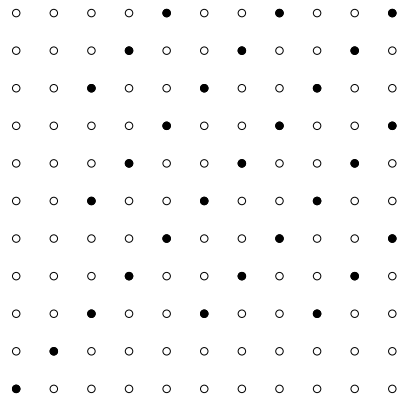
\begin{center}
\vspace{0.5em}
\scalebox{0.5}{\begin{tikzpicture}
\foreach \x in {0,1,...,10}{
      \foreach \y in {0,1,...,10}{
        \node[draw,circle,inner sep=2pt,color=black] at (\x,\y) {};
      }
    }

\foreach \i in {0,1,...,10}{
    \node[draw,circle,inner sep=2pt,color=black,fill] at (\i,\i) {}; 
}
    
\foreach \i in {0,1,...,5}{
    \node[draw,circle,inner sep=2pt,color=black,fill] at (2+\i,5+\i) {};
    \node[draw,circle,inner sep=2pt,color=black,fill] at (5+\i,2+\i) {};
}

\foreach \i in {0,1,2}{
    \node[draw,circle,inner sep=2pt,color=black,fill] at (2+\i,8+\i) {};
    \node[draw,circle,inner sep=2pt,color=black,fill] at (8+\i,2+\i) {};
}
\end{tikzpicture}}
\vspace{0.5em}
\end{center}
Here the grid consists of the elements $(i,j)$ with
\begin{align*}
&x \leq i \leq x + 10 \\
&y \leq j \leq y + 10
\end{align*}
The elements of the grid that are contained in the minimal open set, are indicated by a black dot.
\item Take $N = \NN$, $Z = \ZZ$, and congruence $\sim$ generated by
\begin{equation*}
k \sim 1, \qquad \text{for all }k \geq 1.
\end{equation*}
Again, $G_1 = \ZZ \times \ZZ$ and $G_0 = \ZZ$. The topology on $G_0$ has as open sets the upwards closed sets for the usual partial order on $\ZZ$. The basic open set $\Pi_{(x,y)} \subseteq G_1$ can be drawn as:
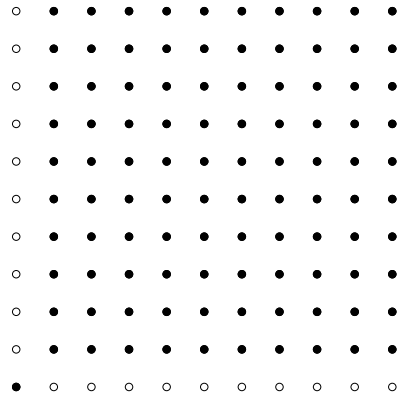
\begin{center}
\vspace{0.5em}
\scalebox{0.5}{\begin{tikzpicture}
\foreach \x in {0,1,...,10}{
      \foreach \y in {0,1,...,10}{
        \node[draw,circle,inner sep=2pt,color=black] at (\x,\y) {};
      }
    }
    
\foreach \x in {1,...,10}{
      \foreach \y in {1,...,10}{
        \node[draw,circle,inner sep=2pt,color=black,fill] at (\x,\y) {};
      }
    } 
\node[draw,circle,inner sep=2pt,color=black,fill] at (0,0) {}; 
\end{tikzpicture}}
\vspace{0.5em}
\end{center}
Again the grid consists of the elements $(i,j)$ with
\begin{align*}
&x \leq i \leq x + 10 \\
&y \leq j \leq y + 10
\end{align*}
The elements of the grid that are contained in the minimal open set, are indicated by a black dot.
\end{enumerate}
\end{example}

\section*{Acknowledgments}

I would like to thank Karin Cvetko-Vah and Lieven Le Bruyn for the interesting discussions regarding the interpretation of \'etale spaces (over an Alexandrov-discrete space) in terms of posets.

\bibliographystyle{amsalphaarxiv}
\bibliography{thesis}
\end{document}